\documentclass[12pt]{article}

\usepackage[english]{babel}
\usepackage[utf8]{inputenc} 
\usepackage[T1]{fontenc}
\usepackage[francais]{layout}
\usepackage[babel=true]{csquotes}

\usepackage{vmargin}
\setpapersize{A4}
\setmarginsrb{	15mm}{	15mm}{	15mm}{	15mm}{	0mm}{	5mm}{	5mm}{	5mm}

\usepackage{amsmath}
\usepackage{amssymb}
\usepackage{MnSymbol}   
\usepackage{stmaryrd}   
\usepackage{amsthm}

\theoremstyle{definition}
\newtheorem{thm}                        {Theorem}
\newtheorem*{thmA}                    {Theorem A}
\newtheorem*{corB}                  {Corollary B}
\newtheorem*{corC}                  {Corollary C}
\newtheorem*{thmD}                    {Theorem D}
\newtheorem*{corE}                  {Corollary E}
\newtheorem*{prpF}                {Proposition F}
\newtheorem{cor}   [thm]              {Corollary}
\newtheorem{defi}  [thm]             {Definition}
\newtheorem{defis} [thm]            {Definitions}
\newtheorem{dfpp}  [thm] {Definition-Proposition}

\newtheorem{lem}   [thm]                  {Lemma}

\newtheorem{notas} [thm]              {Notations}
\newtheorem{prp}   [thm]            {Proposition}
  \newcommand{\A}       {\mathcal A}
  \newcommand{\B}       {\mathcal B}
  \newcommand{\C}       {\mathcal C}
  
\renewcommand{\d}         {\partial}
  \newcommand{\De}          {\Delta}
  \newcommand{\F}       {\mathcal F}
  \newcommand{\G}       {\mathcal G}
  
  \newcommand{\Ga}          {\Gamma}
  \newcommand{\io}           {\iota}
  \newcommand{\K}       {\mathcal K}
  \newcommand{\la}         {\lambda}
  \newcommand{\La}         {\Lambda}
  \newcommand{\N}        {\mathbb N}
  \newcommand{\Ne}      {\mathcal N}
\renewcommand{\O}       {\mathcal O}
  \newcommand{\Om}          {\Omega}
\renewcommand{\P}       {\mathcal P}
  \newcommand{\Sr}       {\mathcal S}
  \newcommand{\ta}      {\mathcal T}
  \newcommand{\U}       {\mathcal U}
  \newcommand{\V}       {\mathcal V}
  \newcommand{\W}       {\mathcal W}
  \newcommand{\ra}                 {\rightarrow}
  
  \newcommand{\LR}             {\Leftrightarrow}
  \newcommand{\ma}                     {\mapsto}
  
  \newcommand{\vi}                 {\varnothing}
\renewcommand{\le}                   {\leqslant}
\renewcommand{\ge}                   {\geqslant}
  \newcommand{\subs}                 {\subseteq}
  \newcommand{\subsn}               {\subsetneq}
  \newcommand{\nsubs}               {\nsubseteq}
  \newcommand{\sups}                 {\supseteq}

  \newcommand{\ti}                      {\times}
  \newcommand{\sm}                   {\setminus}
  \newcommand{\sans}[1]    {\setminus \{ #1 \} }
  \newcommand{\fa}                     {\forall}
  \newcommand{\ex}                     {\exists}
  \newcommand{\se}[1] {\llbracket #1 \rrbracket}    
  \newcommand{\et}                {~\text{and}~}
  \newcommand{\Sm}                {\text{Small}}
  \newcommand{\ci}                       {\circ}
  \newcommand{\po}                       {\cdot}

  \newcommand{\ad}                   {\overline}    
\renewcommand{\int}                  {\mathring}
  \newcommand{\gr}[1]              {\textbf{#1}}
\renewcommand{\it}[1]              {\textit{#1}}

\newcommand{\fonct}[5]
       {\begin{array}{|rccl} #1 : & #2 & \rightarrow & #3 \\ & #4 & \mapsto & #5 \end{array}}
\newcommand{\fonctsn}[4]
       {\begin{array}{|rcl}         #1 & \rightarrow & #2 \\   #3 & \mapsto & #4 \end{array}}
\newcommand{\tend}[1]
       {\underset {#1 \rightarrow \infty}              {\longrightarrow}}

\newcommand{\clu}[5]
       {{#1}_{|(#3 \setminus \{ #4 \})} \underset {#2 \rightarrow \infty}
                                {- \!\! - \!\!\! \twoheadrightarrow} #5 }       

\begin{document}

\title {Convergence actions and Specker compactifications}
\date{\today}
\author{Clément Toromanoff}
\maketitle

\subsection*{Abstract}

 The aim of this paper is to make links between the Specker compactifications of a locally compact group and its convergence actions on compacta. If $G$ is a convergence group acting on a compactum $T$ we prove that the closure of $G$ in the attractor sum $G \sqcup T$ \cite{GeFMRHG} is a quasi-Specker compactification of $G$. Together with a theorem due to H. Abels \cite{AbSK}, this implies that for any convergence action such that the limit set $\La$ is totally disconnected there exists a surjective $G$-equivariant continuous map ${\text{Ends } \! G \ra \La}$. Conversely, when the group is compactly generated we show that any Specker compactification gives rise to a convergence action. Given two minimal convergence actions of a compactly generated group on totally disconnected compacta $M$ and $N$, we can then prove the existence of a minimal compactum $T$ admitting a convergence action of $G$ and $G$-equivariant continuous maps $T \ra M$ and $T \ra N$. We end the paper giving an interpretation of the set of ends of a compactly generated group as the completion by a \enquote{visibility} uniformity \cite{GeFMRHG}.

\section{Introduction}

 Convergence groups were introduced by F.W. Gehring and G.J. Martin \cite{GeMaDQG}. The convergence property that defines these groups reflects an essential dynamical behaviour of a kleinian group acting on the boundary of a hyperbolic space. Wider families of groups satisfy this property, for instance discrete hyperbolic (resp. relatively hyperbolic) groups can be defined as \it{uniform} (resp. \it{expansive}) convergence groups (see \cite{BoTCHG,GeECGRH,TuCLPUCG} and references therein).

 We define here the convergence property for a locally compact group acting on a compact Hausdorff space (\it{compactum} for short). Locally compact hyperbolic groups can also be characterized as uniform convergence groups (see \cite{CaDrLCCGNTA}). It seems to be interesting to ask which definition of a relatively hyperbolic group would fit the best to the locally compact case and whether one can extend results such as in \cite{GeECGRH,TuCLPUCG,YaTCRHG}.

 We refer to the introduction of \cite{GeECGRH} for a history of convergence groups and motivations for the locally compact case. More results about discrete convergence groups can be found in \cite{BoCGCS,FrePCGS,GeMaDQG,TuCGGHS}.

\medskip

 Let us fix a (always Hausdorff) locally compact non compact topological group $G$ acting continuously on a compactum $T$. We do not assume that $T$ is metrizable so that we need to work with nets rather than sequences, see section \ref{s2} for precise definitions. We identify $G$ with a subgroup of homeomorphisms of $T$ and we assume that $|T| > 2$.

 A net on $G$ is \it{wandering} if it has no subnet converging to an element of $G$. A net $(g_n)_n$ on $G$ \it{collapses} (to $a$ except at $b$, where $a,b \in T$) if for every neighborhood $V$ of $a$ in $T$ and every compact subset $K \subs T \sans b$ we have $g_n \po K \subs V$ for $n$ large enough. In other words, a net is collapsing if it converges to a constant map uniformly on compact subsets of $T$ minus one point. Note that $a$ and $b$ need not be distinct.

 Following \cite{BoCGCS}, we say that the action of $G$ on $T$ is \it{convergence} (or that $G$ is a \it{convergence group}) if every wandering net on $G$ has a collapsing subnet. The action is convergence if and only if the action on the space on distinct triples is \it{proper}, see section \ref{s3}. The \it{limit set} $\La$ of a convergence action is the set of all points $a \in T$ such that there exists a net on $G$ collapsing to $a$, it is a $G$-invariant closed subset of $T$.

\medskip

 We show in this paper that the convergence actions of $G$ are closely related to its \it{Specker compactifications} in the sense of H. Abels \cite{AbSK}. A \it{quasi-Specker compactification} of $G$ is a compactum $X$ satisfying the following properties:

\begin{itemize}
 \item $X$ contains $G$ as a dense open subset,
 \item the map $\fonct{\hat R}
     {G \ti X}        {X}
      {(g,x)}  {\left\{ \begin{array}{ccl}
                x \po g & \text{if} & x \in G \\
                   x    & \text{if} & x \in X \sm G \\ \end{array} \right. }$ is continuous (it extends the right multiplications by elements of $G$ by the identity on $X \sm G$),
 \item for all $g \in G$, the left multiplication $\fonct{L_g}{G}{G}{h}{g \po h}$ extends to a continuous map $\hat L_g : X \ra X$.
\end{itemize}

 When in addition the \enquote{boundary} $X \sm G$ is \it{totally disconnected} (i.e. its only connected subsets are the singletons) we say that $X$ is a \it{Specker compactification} of $G$.

 Note that this definition does not depend on the action of $G$ on $T$, note also that the map $\fonct{\hat L}{G \ti X}{X}{(g,x)}{\hat L_g (x)}$ is an action of $G$ on $X$ extending the left multiplications by elements of $G$. In the cases we will be interested in this action will be continuous.

\medskip

 Assume that $G$ is a convergence group acting on a compactum $T$, we denote by $\po$ both the multiplication of $G$ and the action of $G$ on $T$. V. Gerasimov introduced in \cite{GeFMRHG} a topology on the disjoint union $X = G \sqcup T$, called \it{attractor sum}, such that $X$ is a compactum and to which there is a natural extension of the convergence action of $G$. Our first result is:

\begin{thmA}[Theorem \ref{thmB}]
 Let $G$ be a locally compact topological group acting on a compactum $T$ with convergence property and let $X = G \sqcup T$ be the attractor sum. Assume that $|T| \neq 2$.

 The map $\fonct{\hat L}{G \ti X}{X}{(g,x)}{g \po x}$ is continuous, and the closure of $G$ in $X$ is a quasi-Specker compactification of $G$ (with $\hat L_g$ the restriction of $\hat L (g, \po)$ for all $g \in G$).
\end{thmA}

 As a set, the closure of $G$ in the attractor sum is the disjoint union of $G$ and the limit set $\La$ (see above). When the limit set is totally disconnected, we obtain then a (complete) Specker compactification of $G$.

\medskip

 Let $G$ be an (abstract) locally compact topological group. If $G$ acts on two compacta $M$ and $N$, we call \it{$G$-morphism} a $G$-equivariant continuous map $M \ra N$.

 By \cite[Satz 1.8]{AbSK} the \it{end compactification} $X_E$ of $G$ (in the sense of \cite{21}), for which the boundary $E = X_E \sm G$ is the \it{space of ends} of $G$, is a Specker compactification. It is universal, i.e. for any other Specker compactification $X$ of $G$ there exists a surjective $G$-morphism (for the $\hat L$ actions) $f: X_E \ra X$ extending the identity on $G$. Together with Theorem A, this implies the following:

\begin{corB}[Proposition \ref{cor12}]
 For every convergence action of $G$ on a compactum $T$ such that the limit set $\La$ is totally disconnected, there exists a surjective $G$-morphism $E \ra \La$.
\end{corB}

 Let us fix a convergence action of $G$ on $T$ with limit set $\La$. We consider the equivalence relation on $\La$ \enquote{belong to the same connected component} and denote with $\La'$ the associated quotient space, it is a totally disconnected compactum. We show that the natural induced action of $G$ on $T' = (T \sm \La) \sqcup \La'$ is convergence with limit set $\La'$ (Lemma \ref{qcc}). Corollary B implies then the following generalization of a result of R. S. Kulkarni \cite{KuGDD} to the locally compact case (note that we do not require here the group to be finitely or compactly generated):

\begin{corC}[Proposition \ref{cor12}]
 If a locally compact topological group admits a convergence action such that the limit set has more than two components then the group has infinitely many ends.
\end{corC}

 Our next result deals with \it{compactly generated} groups:

\begin{thmD}[Theorem \ref{thmC+}]
 Let $G$ be a locally compact compactly generated topological group. If $X$ is a Specker compactification of $G$ then the action $\hat L$ of $G$ on $X$ is continuous, convergence and its limit set is $X \sm G$.
\end{thmD}

  Let us consider the family of all totally disconnected spaces which are limit sets of a convergence action of $G$. When $G$ is compactly generated Theorem D strengthens Corollary B, stating that this family has a maximal element, the set of ends of $G$.

\medskip

 Let $G$ act with convergence property on two compacta $M$ and $N$. Following V. Gerasimov and L. Potyagailo \cite{GeLPSRHAG}, we call \it{pullback} of these actions a triple $(T,\pi_M,\pi_N)$ such that $T$ is a compactum on which $G$ acts with convergence property and $\pi_M: \La_T \ra \La_M$ and $\pi_N: \La_T \ra \La_N$ are $G$-morphisms. A pullback $(T,\pi_M,\pi_N)$ is \it{minimal} if for any other pullback $(T',\pi'_M,\pi'_N)$ there exists a $G$-morphism $\phi: T' \ra T$ such that $\pi_M \ci \phi = \pi'_M$ and $\pi_N \ci \phi = \pi'_N$.

 Using our previous results we show the following (recall that a convergence action of $G$ on a compactum $T$ is \it{minimal} if the limit set is the whole $T$):

\begin{corE}[Theorem \ref{thmD}]
 Any two minimal convergence actions of a locally compact compactly generated topological group on totally disconnected compacta containing more than two points admit a minimal pullback.
\end{corE}

 Note that Theorem D and Corollary E are not true if we omit the totally disconnectedness property (or if we replace \enquote{Specker} by \enquote{quasi-Specker}). We also need to assume that the group is compactly generated. Counterexamples are due to \cite{BaRiCTMDNAE} and \cite{GeLPSRHAG}, see the discussion at the end of section \ref{s7}.

\medskip

 We conclude our paper giving an interpretation of the space of ends of a locally compact compactly generated group as a \enquote{visibility completion} .

 Recall that the notion of the ends of a topological space or group is originally due to H. Freudenthal \cite{FreuUETRG,FreuNdE}. J. R. Stallings gives in \cite{StGTTDM} a construction of the space of ends of a finitely generated group $G$ (or of a connected locally finite graph) as the set of maximal ideals of a Boolean subalgebra of the power set $\P(G)$. H. Abels extends in \cite{AbSK} this construction to any locally compact space or group.

 Any end of $X$ in the sense of H. Freudenthal naturally gives an end in the sense of H. Abels. When $X$ is locally compact, connected and locally connected, this induces in fact a homeomorphism so that the two definitions coincide (it is a consequence of \cite[I,p.117,ex.19]{BouTG}).

\medskip

 Let $\Ga = (V,E)$ be a locally finite connected graph. Following V. Gerasimov \cite[4.1]{GeFMRHG}, we consider the uniformity on $V$ generated by the sets of couples of vertices that can be connected by a path avoiding a fixed finite set of edges. This uniformity is a \enquote{visibility}, see \cite{GeFMRHG}.

 The boundary of the completion by this uniformity can naturally be seen as the set of ends of $\Ga$ in the sense of H. Freudenthal, and V. Gerasimov points out the fact that this is homeomorphic to the set of ends of $\Ga$ in the sense of J. R. Stallings. In particular the two definitions also coincide in this context.

 We show that the same holds for a compactly generated locally compact group (see section \ref{s8} for precise definitions):

\begin{prpF}[Theorem \ref{thmE}]
 Let $G$ be a compactly generated locally compact topological group. The space of ends of $G$ (in the sense of H. Abels) is homeomorphic to the boundary of the completion of $G$ by a \enquote{visibility} uniformity.
\end{prpF}

\subsubsection*{Structure of the paper}

 Section \ref{s2} of our paper provides topological preliminaries, especially the definition and basic properties of the nets. In section \ref{s3} we establish some properties concerning locally compact convergence groups. In section \ref{s5} we give the construction of the attractor sum and prove Theorem A. In section \ref{s6} we recall Abels' classification of the Specker compactifications of a group, we can then define the space of ends and prove Corollaries B and C. In section \ref{s7} we give some properties of the Cayley graph of a compactly generated group and we prove Theorem D and Corollary E. Section \ref{s8} is devoted to the proof of Proposition F.

\subsubsection*{Acknowledgements}

 This paper was written during my doctoral thesis, I am grateful to my director Leonid Potyagailo for patiently listening to me, reading preliminaries versions of the paper and providing numerous corrections. I also want to thank Professor Herbert Abels for inviting me to Bielefeld and the interesting discussions we had. This work was supported in part by the Labex CEMPI (ANR-11-LABX-0007-01).

\section{Preliminaries}\label{s2}

 In the whole paper $|A|$ denotes the cardinality of a set $A$. A subset of a topological space is \gr{bounded} if it is contained in a compact subspace.

 Let $X$ be a topological space. For all $x \in X$, we denote with $\Ne_X (x)$ the family of all (not necessarily open) neighborhoods of $x$ in $X$. Recall that a \gr{local basis} at a point $x \in X$ is a subfamily $\B$ of $\Ne_X (x)$ such that every neighborhood of $x$ contains a member of $\B$. A \gr{compactum} is a compact Hausdorff space. By convention, all locally compact spaces (or groups) that we consider are assumed to be Hausdorff.

\medskip

 It is convenient for us to work with \enquote{nets}, which we define below. This notion is equivalent to the notion of filter (to every net we can associate a filter with the same convergence properties and conversely), we believe it is easier to describe the dynamical properties of convergence using nets.

 We fix a topological space $X$.

\begin{defis}\ 
 \begin{itemize}
  \item A set $D$ is \gr{directed} by a transitive reflexive relation $\ge$ if for all $n,m \in D$ there exists $p \in D$ such that $p \ge n$ and $p \ge m$.
  \item A property $P$ defined of a directed set $(D,\ge)$ is true \gr{for n large enough} if there exists $m \in D$ such that whenever $n \in D$ satisfies $n \ge m$ we have $P(n)$.
  \item A \gr{net} on $X$ is a map $z: D \ra X$ where $D$ is a directed set, we will write it $(z_n)_{n \in D}$ or simply $(z_n)_n$. We also say that the net is \gr{on} a subset $A$ of $X$ if $z_n \in A$ for all $n$.
  \item A net $(z_n)_n$ on $X$ \gr{converges} to a point $x \in X$ (we write then $z_n \tend{n} x$) if for all $U \in \Ne_X (x)$, $z_n \in U$ for $n$ large enough.
  \item A net $(y_e)_{e \in E}$ is a \gr{subnet} of a net $(z_d)_{d \in D}$ if there exists a map $N: E \ra D$ such that $y = z \circ N$ (i.e. $\fa e \in E, ~y_e = z_{N(e)}$) and such that for all $d \in D$ we have $N(e) \ge d$ for $e$ large enough. We will often not need to explicit the map $N$ and we will write the subnet $(z_e)_e$ rather than $(z_{N(e)})_{e \in E}$.
  \item Let $(z_n)_n$ be a net on $X$. A point $x \in X$ is a \gr{cluster point} of $(z_n)_n$ if there exists a subnet of $(z_n)_n$ converging to $x$. The net $(z_n)_n$ is \gr{wandering} if it has no cluster point.
 \end{itemize}
\end{defis}

 Note that if $(z_n)_n$ is a net on $X$ and if $P$ is a property on elements of $X$, the assertion \enquote{$P(z_n)$ for $n$ large enough} is false if and only if there exists a subnet $(z_i)_i$ such that $P(z_i)$ is false for all $i$.

\medskip

 Recall the following characterizations (widely used in the langage of sequences when the spaces are assumed metrizable for instance), see \cite{KeGT} for instance:

 \begin{itemize}
  \item A map $f: X \ra Y$ (where $Y$ is a topological space) is at a point $x \in X$ if and only if for every net on $X$ converging to $x$, the net $(f(x_n))_n$ converges to $f(x)$.
  \item A subset $A \subs X$ is compact if and only if every net on $A$ has a cluster point in $A$.
 \end{itemize}

 This justify that we repeatedly use formulations such as \enquote{up to passing to a subnet} in our proofs.

\medskip

 Recall that $X$ is \gr{totally disconnected} if its only connected subsets are its singletons, by convention the empty space is not totally disconnected.

\section{Locally compact convergence groups}\label{s3}

 We establish in this section basic facts about locally compact convergence groups acting on a (non necessarily metrizable) compactum.

 We first give the definition of convergence group that we will use. We show then that standard statements concerning the limit set of such an action extend to the locally compact case. We conclude with a lemma that provides convergence actions with totally disconnected limit sets.

\medskip

 Let us fix a locally compact group $G$ acting continuously on a compactum $T$.

\begin{defis}\ 
 \begin{itemize}
  \item A net $(g_n)_n$ on $G$ is \gr{collapsing} if there exists $a,b \in T$ such that for every compact subset $K \subs T \sans{b}$ and every neighborhood $U \in \Ne_T (a)$ we have $g_n (K) \subs U$ for $n$ large enough. In this case we write $\clu{g_n}{n}{T}{b}{a}$.
  \item The group $G$ is a \gr{convergence group}, or the action of $G$ on $T$ is \gr{convergence}, if every wandering net on $G$ has a collapsing subnet.
 \end{itemize}
\end{defis}

 Note that if $G$ is compact or if $|T| \leq 2$ the action is always convergence. When $|T| > 2$, any collapsing net $(g_n)_n$ is wandering and such that the set $\{ g_n ~|~ n \}$ is not bounded.

\medskip

 A more standard definition of convergence is the $3$-properness, that we define now. Let $\ta$ be the locally compact subspace $\{ (x,y,z) \in T^3 ~|~ x \neq y \neq z \neq x \}$e it is the space of \gr{distinct triples} of $T$.

\begin{defi}
 The action of $G$ on $T$ is \gr{proper on triples}, or \gr{$3$-proper}, if for all compact subsets $K$ and $L$ of $\ta$ the set $\{ g \in G ~|~ g \po K \cap L \neq \vi \}$ is bounded.
\end{defi}

 When $G$ is a discrete group, the usual terminology is \enquote{properly discontinuous} instead of \enquote{proper}.

 B. Bowditch proves in \cite[\S 1]{BoCGCS} that if $G$ is discrete then $G$ is a convergence group if and only if its action on $T$ is $3$-properly discontinuous. His arguments extend without difficulty to our situation, we have then:

\begin{prp}
 The action of $G$ on $T$ is convergence if and only if it is $3$-proper.
\end{prp}

 In the following we assume that $G$ is a convergence group.

\begin{defis}\ 
 \begin{itemize}
  \item The action is \gr{proper} at $t \in T$ if there exists $V \in \Ne_T (t)$ such that the set $\{ g \in G ~|~ g \po V \cap V \neq \vi \}$ is bounded. The \gr{ordinary set} $\Om$ is the set of all points where the action is proper.
  \item The \gr{limit set} $\La$ is the set of \gr{attractive points}, i.e.:
$$\La = \{ u \in T ~|~ \ex v \in T, ~\ex (g_n)_n ~\text{net on}~ G ~\text{such that}~ \clu{g_n}{n}{T}{v}{u} \}.$$
  \item The action is \gr{minimal} if $\La = T$.
 \end{itemize}
\end{defis}

\begin{prp}\label{limitset}\ 
 \begin{itemize}
  \item The limit set $\La$ is a closed $G$-invariant subset of $T$ and we have $T = \La \sqcup \Om$.
  \item The restriction of the action of $G$ on $\La$ is convergence and minimal.
  \item If $|\La| \ge 3$ then $\La$ is perfect (i.e. every point is an accumulation point), hence infinite.
 \end{itemize}
\end{prp}

\begin{proof}
 For all $t \in T$, a subset $K \subs T \sans t$ is compact if and only if the set $T \sm K$ is an (open) neighborhood of $t$ in $T$. As the action is continuous, for all $u, v \in T$, for all $g \in G$ and for every net $(g_n)_n$ on $G$, we have $\clu{g_n}{n}{T}{v}{u}$ if and only if $\clu{(g \po g_n)}{n}{T}{v}{g \po u}$, we deduce that $\La$ is $G$-invariant.

 Let us now show that $\La = T \sm \Om$, as $\Om$ is open by construction, $\La$ is closed. If $|T| \leq 2$ we easily get $\La = T$, we assume then that $|T| \geq 3$.

 First, let $\la \in \La$: there exists $\mu \in T$ and a wandering net $(g_n)_n$ on $G$ such that $\clu{g_n}{n}{T}{\mu}{\la}$. We can then pick two distinct points $x,y \in T \sans \mu$ and we have $g_n \po x, g_n \po y \tend n \la$ so that $\la$ is non isolated. Under the hypothesis $|\La| \ge 3$ we can assume $x,y \in \La$, as $g_n \po y, g_n \po y \in \La$ for all $n$ we have the last point.

 Let $V \in \Ne_T (\la)$, there exists a point $v \in V \sans {\la,\mu}$, and as $g_n \po v \tend n \la$, we have $g_n \po V \cap V \neq \vi$ for $n$ large enough. The set $\{ g \in G ~|~ g \po V \cap V \neq \vi \}$ being not bounded we have $\la \in T \sm \Om$.

 Let then $\la \in T \sm \Om$. For every $V \in \Ne_T (\la)$ and every compact subset $K \subs G$ there exists then $v = v (V,K) \in V$ and $g_v \in G \sm K$ such that $g_v \po v \in V$. By construction the net $(g_v)_{(V,K)}$ is wandering so that there exists $u,v \in T$ and a subnet $(g_i)_i$ satisfying $\clu{g_i}{i}{T}{v}{u}$, and as $v_i \tend i v$ and $g_i \po v_i \tend i v$ we deduce that $\la \in \{ u, v \} \subs \La$. We have the first point, the second is then easy.
\end{proof}

 Our next lemma will be used to get convergence actions with totally disconnected limit sets.

 Let $\La'$ be the quotient space of $\La$ by the equivalence relation \enquote{belong to the same connected component}. We extend this relation to $T$ with the equality on $\Om$, we write $T'$ the quotient space and $\pi: T \ra T'$ the canonical projection. Note that $T' = \Om \sqcup \La'$.

\begin{lem}\label{qcc}\ 
 The space $T'$ is a compactum and $\La'$ is a totally disconnected subspace. The formula $g \po \pi (t) = \pi (g \po t)$ induces a well-defined action of $G$ on $T'$. This action is continuous and convergence, with limit set $\La'$.
\end{lem}

 Note that $\pi : T \ra \tilde T$ is continuous, surjective and $G$-equivariant by construction.

\begin{proof}
 It is standard that $\La'$ is a totally disconnected compactum (see \cite{BouTG} for instance), it is a subspace of $T'$ by construction. The space $T'$ is compact as the quotient of a compact space, we show that it is Hausdorff: let $t_1,t_2$ be distinct points in $T'$.

 Assume first that $t_1 \in \Om$. Note that $A = \pi^{-1} (t_2)$ is a compact subset of $T$ not containing $t_1$. As $T$ is a compactum and as $\Om$ is open in $T$, there exists a compact neighborhood $K$ of $t_1$ in $T$ contained in $\Om$ and such that $K \cap A = \vi$. Note then that $K$ is a neighborhood of $t_1$ in $T'$ and that $T' \sm K$ is a neighborhood of $t_2$ in $T'$, disjoint from $K$.

 Assume now that $t_1,t_2 \in \La'$. As $\La'$ is a totally disconnected compactum there exists a closed and open subset $U$ of $\La'$ such that $t_1 \in U$ and $t_2 \notin U$ (see for instance \cite{EnGT}), we write $V = \pi^{-1} (U)$, it is a closed and open subset of $\La$. As $T$ is a compactum, there exists disjoint open subsets $O_1$ and $O_2$ of $T$ such that $V \subs O_1$ and $\La \sm V \subs O_2$. We check then that $\pi^{-1} ( \pi (O_i)) = O_i$ for $i \in \{ 1,2 \}$, so that $\pi (O_1)$ and $\pi (O_2)$ are disjoint neighborhood of $t_1$ and $t_2$ in $T'$.

\medskip
 
 The map $(g,\pi (t)) \ma \pi (g \po t)$ is well-defined because $\La$ is $G$-invariant and because $G$ acts by homeomorphisms so that it acts on the set of components of $\La$. We easily check that it is an action of $G$ on $\tilde T$.

 Let $g \in G$ and $t' \in T'$ and let $(g_n)_n$ (resp. $(t'_n)_n$) be a net on $G$ (resp. on $T'$) such that $g_n \tend n g$ in $G$ (resp. $t'_n  \tend n t'$ in $T'$). We show that the only cluster point of the net $(g_n \po t'_n)_n$ is $g \po t'$, as $T'$ is compact we have then $g_n \po t'_n \tend n g \po t'$ in $T'$. Let then assume that some subnet $(g_i \po t'_i)_i$ converges in $T'$. There exists a net $(t_i)_i$ on $T$ such that $\pi (t_i) = t'_i$ for all $i$, up to passing to a subnet we can assume that $t_i \tend i t$ where $t \in T$. As $\pi$ is continuous we have $t' = \pi (t)$, and as the action of $G$ on $T$ is continuous, we have $g_i \po t_i \tend i g \po t$ so that $g_i \po t'_i = g_i \po \pi (t_i) = \pi (g_i \po t_i) \tend i \pi (g \po t) = g \po \pi (t) = g \po t'$.

 Let now $(g_n)_n$ be a net on $G$, we assume that there exists $a,b \in T$ satisfying $\clu{g_n}{n}{T}{b}{a}$ and we show that $\clu{g_n}{n}{T'}{\pi (b)}{\pi (a)}$. Let $V \in \Ne_{T'} (\pi (a))$ and let $K$ be a compact subset of $T' \sans {\pi (b)}$. We have $\pi^{-1} (V) \in \Ne_T (a)$ and $\pi^{-1} (K)$ is closed hence compact and contained in $T \sans b$ so that $g_n \po (\pi^{-1} (K)) \subs \pi^{-1} (V)$ for $n$ large enough. As $\pi$ is surjective we deduce that $g_n \po K = g_n \po \pi (\pi^{-1} (K)) = \pi (g_n \po \pi^{-1} (K)) \subs \pi (\pi^{-1} (V)) \subs V$ for $n$ large enough.

 In other words, collapsing nets \enquote{relatively to $T$} are also collapsing \enquote{relatively to $T'$} so that the convergence property for the action on $T$ gives the convergence property for the action on $T'$. We also deduce that the limit set of the convergence action of $G$ on $T'$ is $\La'$, we have the lemma.
\end{proof}

\section{Attractor sum}\label{s5}

 We show in this section that any convergence action induces a quasi-Specker compactification. The main tool that we use in this prospect is the \gr{attractor sum} construction of V. Gerasimov \cite[\S 8]{GeFMRHG}.

\medskip

 Let us fix again a locally compact group $G$ acting with convergence property on a compactum $T$.

 We first define a topology on the union $X = G \sqcup T$ which extends the original topologies of $G$ and $T$, it is the attractor sum. We then show that the space $X$ is a compactum (Proposition \ref{convAS}) and that the right and left translations extend on $X$ as required (Theorem \ref{ASrlc}), this yields Theorem A. We conclude the section with a short proof that the action of $G$ on the attractor sum is convergence (Proposition \ref{ASc}) and a lemma allowing to extend $G$-equivariant maps on attractor sums.

 The fact that the action of $G$ on the compactum $X$ is convergence is already proven in \cite{GeFMRHG}, we believe that our methods using (collapsing) nets give an other interesting point of view on the attractor sum.

\medskip

 We denote with $\po$ both the multiplication in $G$ and the action of $G$ on $T$, so that $g \po x$ makes sense for $x \in G$ or $x \in T$ (where $g \in G$).

 For any net $(g_n)_n$ on $G$ and any subset $H \subs G$, let
 $$\begin{array}{cccccc}
            & \A (g_n)_n & = & \{ u \in T ~|~ \ex v \in T, & \!\! \ex (g_i)_i ~\text{subnet of}~ (g_n)_n: & \!\! \clu{g_i}{i}{T}{v}{u} \} \\
 \text{and} &   \A(H)    & = & \{ u \in T ~|~ \ex v \in T, & \!\! \ex (g_i)_i ~\text{net on}~ H: & \!\! \clu{g_i}{i}{T}{v}{u} \}
   \end{array}$$
 be the sets of \gr{attractor points} of $(g_n)_n$ and $H$ in $T$.

\begin{dfpp}
 Let $X$ be the formal disjoint union $X = G \sqcup T$. We say that a subset $P \subs X$ is closed if $P \cap G$ is closed in $G$, $P \cap T$ is closed in $T$ and $\A (P \cap G) \subs P \cap T$.

 The family of closed subsets induces the structure of a topological space on $X$. We call this space the \gr{attractor sum} of $G$ and $T$ (relatively to the action). In this space $G$ is an open subset and $T$ is a closed subset, moreover the topologies of $G$ and $T$ as subspaces coincide with their initial ones.
\end{dfpp}

The proof is straightforward. Note that the case $|T| = 1$ gives the Alexandroff one-point compactification of $G$. In the case $|T| = 2$, for all $H \subs G$ we have $\A(H) \in \{ \vi, T \}$ and the attractor sum is the Alexandroff compactification with a \enquote{doubled} point at infinity. These two points are closed but not separated in $X$.

 In the following, we assume then that $|T| \geq 3$.

\begin{prp}\label{convAS}\ 
 \begin{itemize}
  \item The attractor sum $X = G \sqcup T$ is a compactum.
  \item For every net $(g_n)_n$ on $G$ and for all $t \in T$, we have $g_n \tend{n} t$ in $X$ if and only if $(g_n)_n$ is wandering and $\A (g_n)_n = \{ t \}$.
 \end{itemize}
\end{prp}

 The second point gives the two following facts:
\begin{itemize}
 \item[-] if a net $(g_n)_n$ on $G$ and if $t,u \in T$ satisfy $\clu{g_n}{n}{T}{u}{t}$ then $g_n \tend n t$ and $g_n^{-1} \tend n u$ in $X$,
 \item[-] the set of accumulation points of $G$ in $T$ is exactly the limit set of the action.
\end{itemize}

\begin{proof}

We first show that $X$ is Hausdorff, using the same proof as V. Gerasimov \cite[Prop. 8.2.4]{GeFMRHG}. Let $p$ and $q$ be distinct points of $X$, we distinguish three cases:

 \begin{itemize}
  \item Case 1: $p,q \in G$. As $G$ is Hausdorff and open in $X$ we can find disjoint neighborhoods of $p$ and $q$ in $G$ hence in $X$.

  \item Case 2: $p \in G$ and $q \in T$. Let us pick a compact $K \in \Ne_G (p)$, then $K$ is a neighborhood of $p$ in $X$, closed in $X$ because $\A (K) = \vi$, hence $X \sm K$ is a neighborhood of $q$ in $X$ disjoint to $K$.

  \item Case 3: $p,q \in T$. We need some preliminaries.

  Let us fix a subset $\xi \subs T$ of cardinality $3$. For every subset $A \subs T$ let $\xi A = \{ g \in G ~|~ |A \cap g \po \xi| \ge 2 \}$ and $\xi^+ A = A \cup \xi A$. We easily check that $\xi$ and $\xi^+$ commutes with the complement (in $G$ or in $X$ respectively) and that for all disjoint subsets $A$ and $B$ of $G$ the subsets $\xi^+ A$ and $\xi^+ B$ of $X$ are still disjoint. Let $A$ be a closed subset of $T$, then $T \cap \xi^+ A = A$ is closed in $T$. Let $(g_n)_n$ be a net on $G \cap \xi^+ A = \xi A$ converging to a point $g \in G$, we show that $g \in \xi A$, as a consequence $\xi A$ is closed in $G$. For all $n$ we have $|A \cap g_n \po \xi| \ge 2$, as $\xi$ is finite and $A$ is compact, up to passing to a subnet we can assume that there exists distinct points $x_1$ and $x_2$  in $\xi$ such that $g_n \po x_1 \in A$ and $g_n \po x_2 \in A$ for all $n$, and that there exists points $a_1, a_2 \in A$ such that $g_n \po x_1 \tend{n} a_1$ and $g_n \po x_2 \tend{n} a_2$. The action of $G$ on $T$ is continuous so that $g \po x_1 = a_1 \in A$ and $g \po x_2 = a_2 \in A$, hence $g \in \xi A$. Finally, if there exists $u \in \A (\xi A) \sm A$, there exists $v \in T$ and a net $(g_n)_n$ on $\xi A$ such that $\clu{g_n}{n}{T}{v}{u}$, and as $X \sm A \in \Ne_X (u)$, we have $g_n (\xi \sans v) \subs X \sm A$  for $n$ large enough, but $|\xi \sans v| \ge 2$ and $g_n \in \xi A$ for all $n$ gives a contradiction. We deduce that $\A (\xi A) \subs A$ and we conclude that $\xi^+ A$ is closed in $X$.

 Now, if $O_p$ and $O_q$ are disjoint open neighborhoods of $p$ and $q$ in $T$ then $\xi^+ O_p$ and $\xi^+ O_q$ are disjoint open subsets of $X$ containing $p$ and $q$.

 \end{itemize}

 We show now the second point: let $(g_n)_n$ be a net on $G$ and let $t \in T$.

 We first assume that $g_n \tend n t$ in $X$. The net $(g_n)_n$ (seen as a net on $G$) is wandering, by the convergence property we have $\A (g_n)_n \neq \vi$. For $r \in T \sans t$, there exists $O_r \in \Ne_X (r)$ open and such that $F_r = X \sm O_r$ belongs to $\Ne_X (t)$. On one hand $F_r$ is closed in $X$ so that $\A (G \cap F_r) \subs T \cap F_r$ (which does not contain $r$), one the other hand $g_n \tend{n} t$ in $X$ so that $g_n \in G \cap F_r$ for $n$ large enough. We deduce that $r \notin \A (g_n)_n$, and necessarily $\A (g_n)_n = \{ t \}$.

 Conversely, assume that $(g_n)_n$ is wandering and that $\A (g_n)_n = \{ t \}$. For every $O \in \Ne_X (t)$ open, the complement $F = X \sm O$ is closed in $X$ so that $\A (G \cap F) \subs T \cap F$ (which does not contain $t$), hence there is no subnet of $(g_n)_n$ on $G \cap F$. This means exactly that $g_n \in G \sm F \subs O$ for $n$ large enough, so that $g_n \tend {n} t$ in $X$.

\medskip

 We can now show that $X$ is compact. Let $(x_n)_n$ be a net on $X$, up to passing to a subnet, we limit the proof to the three following cases. First case: $x_n \in T$ for all $n$. As $T$ is compact the net has a cluster point in $T$ and it is also a cluster point in $X$. Second case: $x_n \in G$ for all $n$ and the net has a cluster point in $G$. Again, this is a cluster point in $X$. Third case: $x_n \in G$ for all $n$ and the net is wandering. By the convergence property there exists a subnet $(x_i)_i$ and points $u, v \in T$ such that $\clu{x_i}{i}{T}{v}{u}$, by the previous point we have $x_i \tend{i} u$ in $X$. We have shown that every net on $X$ has a cluster point: $X$ is compact.
\end{proof}

\begin{thm}\label{ASrlc}
 Assume that $G$ is a topological group.

 The actions $\fonct{\hat R}
     {G \ti X}        {X}
      {(g,x)}  {\left\{ \begin{array}{ccl}
                x \po g & \text{if} & x \in G \\
                   x    & \text{if} & x \in T \\ \end{array} \right. }$
and $\fonct{\hat L}{G \ti X}{X}{(g,x)}{g \po x}$ are continuous.
\end{thm}

\begin{proof}
 The two maps are actions of $G$ on $X$ by construction. We only prove the continuity of $\hat R$, the case of $\hat L$ is similar.

 It is enough to show the continuity at a point $(g,t) \in G \ti T$. Let $(g_n)_n$ be a net on $G$ and $(x_n)_n$ be a net on $X$ such that $g_n \tend{n} g$ in $G$ and $x_n \tend{n} t$ in $X$, we show that $\hat R (g_n,x_n) = x_n \po g_n \tend{n} \hat R (g,t) = t$ in $X$.

 Using the partition $\{ n ~|~ x_n \in T \} \sqcup \{ n ~|~ x_n \in G \}$ and noting that $\hat R_{|(G \ti T)}$ is continuous, it is enough to consider the case where $x_n \in G$ for all $n$. By the previous proposition the net $(x_n)_n$ is wandering and satisfies $\A (x_n)_n = \{ t \}$, we want to show that the net $(x_n \po g_n)_n$ is wandering and has only one attractor point, namely $t$.

\medskip

 By contradiction, assume that there exists a subnet $(x_i \po g_i)_i$ converging in $G$ to a point $h \in G$. As $g_i \tend{i} g$ in $G$ we have $g_i^{-1} \tend{i} g^{-1}$ in $G$ so that $x_i = (x_i \po g_i) \po g_i^{-1} \tend{i} h \po g^{-1}$ in $G$ hence in $X$, this is in contradiction with $x_i \tend{i} t \in T$. We deduce that the net $(x_n \po g_n)_n$ is wandering.

 Let then $u,v \in T$ and $(x_i \po g_i)_{i \in I}$ be a subnet such that $\clu{(x_i \po g_i)}{i}{T}{v}{u}$, we show that $\clu{x_i}{i}{T}{g \po v}{u}$ so that necessarily $u = t$. Let then $K$ be a compact subset of $T \sans {g \po v}$ and $V \in \Ne_T u$.

 We first claim that there exists a neighborhood $W_0 \in \Ne_T (v)$ such that $g_i^{-1} \po K \subs T \sm W_0$ for $i$ large enough. Assume that this is not true: for all $W \in \Ne_T (v)$ and for all $i \in I$, there exists then $j = j (i,W) \ge i$ and $k_j \in K$ such that $g_j^{-1} \po k_j \in W$. The map $(g_j^{-1} \po k_j)_{(i,W)}$ is then a net on $T$, converging to $v$ by construction. As $K$ is compact, some subnet $(k_{j'})_{j'}$ of $(k_j)_{(i,W)}$ converges to a point $k \in K$ distinct from $g \po v$. But the action of $G$ on $T$ is continuous so that the net $(g_{j'}^{-1} \po k_{j'})_{j'}$ converges to $g^{-1} \po k \neq v$, a contradiction.

 Without loss of generality, we can assume that $T \sm W_0$ is a compact subset of $T \sans v$. As $\clu{(x_i \po g_i)}{i}{T}{v}{u}$, for $i$ large enough we have $x_i \po K = (x_i \po g_i) \po (g_i^{-1} \po K) \subs (x_i \po g_i) \po (T \sm W_0) \subs V$. We have the stated result, we deduce that $\A (x_n \po g_n)_n = \{ t \}$.

\end{proof}

 We can now prove Theorem A:

\begin{thm}\label{thmB}
 If $G$ is a locally compact convergence group acting on a compactum $T$ containing at least three points then the closure of $G$ in the attractor sum of $G$ and $T$, endowed with the restriction of the $\hat L$-action defined previously, is a quasi-Specker compactification of $G$ (see the introduction or Definitions \ref{dSc}).
\end{thm}

 We will use the following corollary:

\begin{cor}\label{corthmB}
 Let $G$ be a locally convergence group acting on a compactum $T$ containing at least three points with a totally disconnected limit set. The closure of $G$ in the attractor sum of $G$ and $T$ is a (complete) Specker compactification of $G$.
\end{cor}

\begin{proof}
 Let us write $\bar G$ the closure of $G$ in the attractor sum $X$ of $G$ and $T$. As $X$ is a compactum so is $\bar G$, and it contains $G$ as a dense subset by definition. As a set $\bar G$ is the disjoint union of $G$ and the limit set $\La$ by proposition \ref{convAS}. As $\La$ is $G$-equivariant the maps $\hat L$ and $\hat R$ restrict to maps $G \ti \bar G \ra \bar G$ which are continuous and extend the multiplications of $G$.

 As the third point of Definitions \ref{dSc} only depends on the topology of the boundary, we get the corollary.
\end{proof}

 Note that by Lemma \ref{qcc} any convergence action induces a convergence action with a totally disconnected limit set, hence a Specker compactification 

 Note also that Theorems \ref{ASrlc} and \ref{thmB} are trivially true if $|T| = 1$, if $|T| = 2$ we still have Theorem \ref{ASrlc} (in the attractor sum, every wandering net on $G$ converges to any point in $T$, and any non wandering net has all its cluster points in $G$).

\medskip

 We now give a short proof of \cite[8.3]{GeFMRHG}:

\begin{prp}\label{ASc}
 The action of $G$ on the attractor sum $X$ of $G$ and $T$ is convergence.
\end{prp}

\begin{proof}
 Let $(g_n)_n$ be a wandering net on $G$, we can assume that there exists $u,v \in T$ such that $\clu{g_n}{n}{T}{v}{u}$ (we have then $g_n \tend n u$ in $X$, see Proposition \ref{convAS}). We claim that $\clu{g_n}{n}{X}{v}{u}$. 

 Arguing by contradiction, assume that this is not the case. There exists then a compact subset $K \subs X \sans v$, a neighborhood $U \in \Ne_X (u)$, a subnet $(g_i)_i$ and a net $(x_i)_i$ on $K$ such that $g_i \po x_i \notin U$ for all $i$. We can assume that there exists a point $x \in K \subs X \sans v$ and a point $y \in X \sans u$ such that $x_i \tend i x$ and $g_i \po x_i \tend i y$ in $X$.

 Assume first that $x \in G$. We have then $x_i \in G$ for $i$ large enough hence $x_i \tend i x$ in $G$, but $g_i \tend i u$ in $X$ and the map $\hat R$ is continuous so that  $g_i \po x_i = \hat R (x_i, g_i) \tend i \hat R (x,u) = u \neq y$ in $X$, a contradiction.

 We deduce that $x \in T$, and as $g_i^{-1} \tend i v$ in $X$ the same argument applied to $x_i = g_i^{-1} \po (g_i \po x_i)$ gives $y \in T$.

 Up to passing to a subnet, we can assume either that $x_i \in T$ for all $i$ or that $x_i \in G$ for all $i$. The first case contradicts the assumption $\clu{g_i}{i}{T}{v}{u}$, and in the second case, we can assume that there exists $\la, \mu \in T$ such that $\clu{x_i}{i}{T}{\la}{x}$ and $\clu{g_i \po x_i}{i}{T}{\mu}{y}$. Let us pick a point $p \in T \sans {\la, \mu}$, on one hand we have $(g_i \po x_i) \po p  \tend i y$ because $p \neq \mu$ and on a other hand we have $x_i \po p \tend i x \neq v$ so that $g_i \po (x_i \po p) \tend i u$, once more a contradiction.
\end{proof}

 We will need the following result in section \ref{s7}, it has already been proved in the case of a discrete group \cite[Lemma 5.3]{GeLPSRHAG}.

\begin{lem}\label{cGeAS}
 Let $G$ act with convergence property on two compacta $M$ and $N$ and let $f: M \ra N$ be a non-constant continuous $G$-equivariant map. If $|M| \geq 3$ then the map $\tilde f: G \sqcup M \ra G \sqcup N$ extending $f$ on the attractor sums by the identity on $G$ is continuous.
\end{lem}

 Note that necessarily $|N| \geq 3$.

\begin{proof}
 Under the hypotheses of the lemma we first show the following: if $m_a, m_r \in M$, if $n_a, n_r \in N$ and if $(g_n)_n$ is a wandering net on $G$ such that $\clu{g_n}{n}{M}{m_r}{m_a}$ and $\clu{g_n}{n}{N}{n_r}{n_a}$ then $n_a = f (m_a)$ and $n_r = f (m_r)$.

 Note first that $m_r$ is not isolated in $M$. Indeed, it follows from the definition of a collapsing net that $\clu{g_n^{-1}}{n}{M}{m_a}{m_r}$, as $|M| \geq 3$ we can find $x,y \in M$ such that $|\{ x,y, a \}| = 3$, we have then $g_n^{-1} \po x, g_n^{-1} \po y \tend n m_a$, and $g_n^{-1} \po x \neq g_n^{-1} \po y$ for all $n$.

 Note then that $f^{-1} (n_r) \cup \{ m_r \} \subsn M$. If not, as $m_r$ is not isolated in $M$ we can find a net $(m_n)_n$ on $M$ converging to $m_r$ and such that $m_n \neq m_r$ for all $n$, by assumption $f (m_n) = n_r$ for all $n$, the continuity of $f$ gives then $f (m_r) = n_r$ so that $f$ is constant, a contradiction.

 We can then take $m \in M \sm ( f^{-1} (n_r) \cup \{ m_r \} )$, on one hand $m \neq m_r$ so that $g_n \po m \tend n m_a$ and $f (g_n \po m) \tend n f (m_a)$, on an other hand $f(m) \neq n_r$ so that $g_n \po f(m) \tend n n_a$, but $g_n \po f (m) = f (g_n \po m)$ for all $n$ so that $n_a = f (m_a)$.

 As $\clu{g_n^{-1}}{n}{M}{m_a}{m_r}$ and $\clu{g_n^{-1}}{n}{N}{n_a}{n_r}$ the same proof gives $n_r = f (m_r)$.

 We can now prove the lemma. Easily, $\tilde f$ is continuous on the open subsets $G$ and $\Om_M$ of $M$. Using the same techniques as in the proofs of Theorem \ref{ASrlc} and Proposition \ref{convAS}, we prove that if $(x_n)_n$ is a net on $G \sqcup M$ converging to a point $\la \in \La_M$, then $\tilde f (x_n) \tend n \tilde f (\la) = f (\la)$. It is enough to consider the case when $x_n \in G$ for all $n$, the net $(x_n)_n$ is then wandering and satisfies $\A_M (x_n)_n = \{ \la \}$. If $n_a, n_r \in N$ and if $(x_i)_i$ is a subnet satisfying $\clu{x_i}{i}{N}{n_r}{n_a}$, we can assume that there exists $m_a, m_r \in M$ such that $\clu{x_i}{i}{M}{m_r}{m_a}$, necessarily $m_a = \la$ and the previous result gives then $n_a = f (\la)$. We deduce that $\A_N (x_n)_n = \{ f (\la) \}$ and the lemma.
\end{proof}

\section{Specker Compactifications}\label{s6}

 The aim of this section is to recall Abels' classification of the Specker compactifications of a given group \cite[Satz 1.6 and 1.8]{AbSK} (our Proposition \ref{class}), this yields the definition of the space of ends. We conclude the section with the proof of Corollaries B and C.

 Our methods in the following sections use a similar machinery, in order the paper to be self-contained we provide some details here.

\medskip

 In the whole section we fix a locally compact topological group $G$.

\begin{defis}\label{dSc}
 A topological space $X$ is a \gr{Specker compactification} of $G$ if it satisfies the following properties:
  \begin{enumerate}
   \item $X$ is a compactum,
   \item $G$ is homeomorphic to a dense open subset of $X$ (we will identify $G$ and this subset),
   \item the \enquote{boundary} $\La = X \sm G$ is totally disconnected (its only connected subsets are the singletons),
    \item the map $\fonct {\hat R} {G \ti X} {X} {(g,x)}{\left\{ \begin{array}{ccl}
                x \po g & \text{if} & x \in G \\
                   x    & \text{if} & x \in X \sm G \\ \end{array} \right. }$ is continuous,
                  
    \item for all $g \in G$, the map $\fonct{L_g}{G}{G}{h}{g \po h}$ extends to a continuous map $\hat {L_g} : X \ra X$.
   \end{enumerate}

 A topological space $X$ satisfying the three first points (related only to the topological structure of $G$) is a \gr{zero-dimensional compactification} of $G$. A topological space $X$ satisfying all points but the third is a \gr{quasi-Specker compactification} of $G$.

 Two compactifications $X$ and $Y$ of $G$ are \gr{equivalent} if the identity map $G \ra G$ extends to a homeomorphism $X \ra Y$.
\end{defis}

\begin{notas}\label{nots26}\ 
 \begin{itemize}
  \item For all $A,B \subs G$, let $A + B = (A \sm B) \cup (B \sm A)$ be the symmetric difference of $A$ and $B$.
  \item If $A$ is a subset of a topological space $X$, we write $\d_X A$ the boundary of $A$ for the topology on $X$.
  \item Let $\B$ (resp. $\K$) be the family of all bounded (resp. compact) subsets of $G$.
  \item Let $\U^+ = \{ A \subs G ~|~ \d_G A \in \K \}$. Endowed with symmetric difference and intersection $\U^+$ is a Boolean algebra (with identity elements $0 = \vi$ for $+$ and $1 = G$ for $\cap$) in which $\B$ is an ideal.
  \item Let $\U^G = \{ A \in \U^+ ~|~ \fa K \in \K \sm \{ \vi \} , ~(A \po K) + A \in \B \}$, this is a subalgebra of $\U^+$ containing $\B$ and a $G$-module for the left action of $G$.
 \end{itemize}
\end{notas}

\begin{prp}\label{class}\ 
 The map $X \ma \U (X) = \{ B \cap G ~|~ B \subs X ~\text{and}~ \d_X B \subs G \}$ induces a bijection between equivalence classes of zero-dimensional compactifications of $G$ and subalgebras of $\U^+$ containing $\B$. In this correspondance the (equivalence classes of) Specker compactifications of $G$ correspond exactly to the subalgebras of $\U^G$ containing $\B$ which are also $G$-submodules of $\U^G$.
\end{prp}

\begin{proof}[Sketch of the proposition's proof]
 Let us first fix a zero-dimensional compactification $X$ of $G$, we write $\La = X \sm G$. The key fact to get the classifications is the following lemma:

\begin{lem}\label{blc}
 For all $x \in X$, the family $\W_x = \{ B \subs X ~|~ x \in B \et \d_X B \subs G \}$ is a local basis at $x$.
\end{lem}

 As $G$ is locally compact, the case $x \in G$ is easy. Assume $x \in \La$. As $\La$ is compact and totally disconnected it has a base of closed and open subsets, and as $X$ is a compactum it is normal (i.e. disjoint closed subsets have disjoint neighborhoods), see for instance \cite{EnGT}. It is then an exercise to show the lemma.

\medskip

 Let $\U (X)$ be as in the statement of the proposition, we easily check that it is a subalgebra of $\U^+$ containing $\B$.

 A subfamily $\F$ of $\U(X)$ is an \gr{unbounded ultrafilter in $\U(X)$} if it satisfies the following properties :
 \begin{itemize}
  \item[(F0)] $X \in \F$ (or $\F \neq \vi$),
  \item[(F1)] for all $A,B \in \F$ we have $A \cap B \in \F$,
  \item[(F2)] for all $A,B \in \U(X)$, if $A \in \F$ and $A \subs B$ then $B \in \F$,
  \item[ (U)] for all $A \in \U(X)$, either $A \in \F$ or $G \sm A \in \F$,
  \item[(NB)] $\F \cap \B = \vi$.
 \end{itemize}

 Using the previous lemma we can show that for all $\la \in \La$, the subfamily $\F_{\la} = \{ B \cap G ~|~ B \in \W_{\la} \}$ of $\U (X)$ is an unbounded ultrafilter in $\U (X)$. Moreover, we can show that the map $\la \ma \F_{\la}$ is a bijection between  $\La$ and the set of unbounded ultrafilters in $\U (X)$. We have then characterized points of $\La$ only in terms of the subalgebra $\U (X) \subs \U^+$.

 Up to taking complement in $\P (\U (X))$ it is equivalent to consider prime ideals instead of ultrafilters (see \cite{AbSK}). Our choice is more consistant with section \ref{s8}.

\medskip

Let now $\U$ be any subalgebra of $\U^+$ containing $\B$. Following the previous results, let $\d_\U G$ be the set of unbounded ultrafilters in $\U$ and let $X (\U) = G \sqcup \d_\U G$. Let $\O$ denote the set of all open subsets of $G$, for $A \in \U \cap \O$ we define $A_+ = A \cup \{ \F \in \d_\U G ~|~ A \in \F \}$, let then $\U^\O_+$ be the set of all $A_+$ for $A \in \U \cap \O$. The family $\O \cup \U^\O_+$ is then a base for a topology on $X (\U)$ satisfying all properties of a zero-dimensional compactification of $G$ (see \cite[Satz 1.6]{AbSK} for more details).

 It is then enough to check that $\U (X (\U)) = \U$ and that $X$ and $X (\U (X))$ are equivalent for any zero-dimensional compactification $X$ of $G$ in order to get the first part of the proposition.

\medskip

 Now, let us fix a Specker compactification $X$ of $G$, we set again $\La = X \sm G$.

 If $A = B \cap G$ where $B \subs X$ satisfies $\d_X B \subs G$ and if $g \in G$, we have $g \po A = \hat {L_g} (B) \cap G$ and $\hat {L_g}$ is a homeomorphism so that $\d_X (\hat {L_g} (B)) = \hat {L_g} (\d_X B) \subs G$. We deduce that $\U (X)$ is a $G$-module for the left action. Using topological arguments, one can also show that $\U (X) \subs \U^G$.

\medskip

 Conversely, let  $\U$ be a subalgebra and a $G$-submodule of $\U^G$ containing $\B$. Let $X$ be the corresponding zero-dimensional compactification of $G$ (first part of the proposition), we have then $X \sm G = \d_\U G$.

 For all $g \in G$, we check that the map $\hat {L_g}$ such that $\hat {L_g} (h) = g \po h$ for $h \in G$ and $\hat {L_g} (\F) = \{ g \po A ~|~ A \in \F \}$ for $\F \in \d_\U G$ is continuous. We can also show that the map $\hat R : G \ti Y \ma Y$ (see Definitions \ref{dSc}) is continuous. In conclusion $X$ is a Specker compactification, we have the proposition.
\end{proof}

\begin{prp}\label{ends}
 The Specker compactification $X_E (G) = X (\U^G)$ satisfies the following \enquote{universal} property: for every Specker compactification $X$ of $G$ there exists a (unique) surjective $G$-equivariant continuous map $X_E (G) \ra X$ extending the identity on $G$.
\end{prp}

 It follows that the space $X_E (G)$ is the \gr{ends compactification} in the sense of E. Specker \cite{21}. The space $E (G) = X_E (G) \sm G$ is the \gr{space of ends} of $G$, it is a totally disconnected compactum on which $G$ acts with the $\hat L$-action. We will see in section \ref{s8} that when $G$ is compactly generated, the space of ends $X_E$ of $G$ can be seen as the boundary of a completion of $G$ by a \enquote{visibility} uniformity on $G$.

 Note that the map $X_E (G) \ra X$ given by the proposition restricts to a surjective $G$-equivariant continuous map $E (G) \ra X \sm G$.

\begin{proof}
 We use the same terminology as in the previous proof.

 First note that if $\U_1$ and $\U_2$ are subalgebras and submodules of $\U^G$ such that $\U_1 \subs \U_2$, then the map $\fonctsn{X(\U_2)}{X(\U_1)}{x}{\left\{ \begin{array}{ccl}
                x & \text{if} & x \in G \\
      x \cap \U_1 & \text{if} & x \in \d_{\U_2} G \\ \end{array} \right. }$ is surjective, $G$-equivariant, continuous, and extends the identity on $G$.

 Let then $X$ be a Specker compactification of $G$. Taking $\U_1 = \U(X)$ and $\U_2 = \U^G$ in the previous point we get a map $f: X_E (G) \ra X (\U (X))$, by the previous proof there exists a homeomorphism $\psi: X (\U (X)) \ra X$, the map $\psi \ci f$ is then the desired map $X_E (G) \ra X$.
\end{proof}

 We can now prove Corollary B et C of the introduction:

\begin{prp}\label{cor12}
 Let $G$ be a locally compact convergence group acting on a compactum $T$.
 \begin{itemize}
  \item If the limit set $\La$ is totally disconnected and if $|T| \neq 2$ then there exists a $G$-equivariant surjective continuous map $E(G) \ra \La$.
  \item If the limit set has more than two components then $G$ has infinitely many ends (i.e. $|E(G)| = \infty$).
 \end{itemize}
\end{prp}

\begin{proof}
 Assume first that $\La$ is totally disconnected (then $\La \neq \vi$ by convention). The case $|T| = 1$ is trivial (we have then $\La = T$), and if $|T| \geq 3$ by Corollary \ref{corthmB} the closure $G \cup \La$ of $G$ in the attractor sum $G \cup T$ is a Specker compactification of $G$, by the previous proposition there exists then a surjective map $E(G) \ra \La$.

 Assume now that $\La$ has more than two components. Let $\La'$ be the totally disconnected quotient space of $\La$ where we collapsed connected components to points and $T' = (T \sm \La) \cup \La'$, then $G$ acts on $T'$ with convergence property and with limit set $\La'$ (see Lemma \ref{qcc}). By propositon \ref{limitset} $\La'$ is infinite, by the previous point there exists a surjective map $E(G) \ra \La'$, we conclude that $E(G)$ is also infinite.
\end{proof}

\section{Case of a compactly generated group}\label{s7}

 We want now to prove Theorem D, from which we will deduce Corollary E. The tools that we use are adapted from \cite{AbEReT} and \cite[\S 2 and \S 3]{AbSK}.

 We start with the construction of the Cayley graph of a compactly generated group $G$, we fix then a Specker compactification $X$ and show the convergence property for the action of $G$ on $X$ (Theorem \ref{thmC+}), this yields Theorem D. We conclude with the proof of Corollary E and some considerations about the hypotheses in Theorem D.
 
\medskip

 In the whole section, $G$ is a \gr{compactly generated} locally compact topological group, we write multiplicatively (with $\po$) the group operation.

\begin{lem}
 There exists a compact symmetric neighborhood $F \in \Ne_G (1)$ that generates $G$.
\end{lem}

\begin{proof}
 If $K$ is a compact subset generating $G$ (i.e. $G = \bigcup_{n \in \N} K^n$) by Baire's lemma some $K^n$ has a non-empty interior, we then check that $F = (K \cup K^{-1})^{2n}$ is as in the statement of the lemma.
\end{proof}

\begin{defis}
 The \gr{Cayley graph} of $G$ (relatively to the generating set $F$) is the set (of edges) $\Ga = \{ (x, x \po f) ~|~ x \in G, f \in F \}$. A \gr{$\Ga$-path} connecting two points $g, h \in G$ is a finite sequence $g = x_0, ~x_1, ~\dots~, ~x_n = h$ (where $n \in \N$) of points in $G$ such that $(x_0,x_1), \dots, ~(x_{n-1},x_n) \in \Ga$, we say that the path is \gr{in} a subset $A \subs G$ if all the $x_i$ belongs to $A$. Let $A$ be a subset of $G$. We say that $A$ is \gr{$\Ga$-connected} if any two points in $A$ can be connected by a $\Ga$-path in $A$. Let $\d_\Ga A = \{ (g,h) \in \Ga ~|~ |\{g,h\} \cap A| = 1 \}$ be the \gr{$\Ga$-boundary} of $A$. For all $n \in \N$, let $A^{\le n} = A \po F^n$ be the set of points of $G$ whose \enquote{$\Ga$-distance} from $A$ is at most $n$, it is easy to see that the projection $|\d_\Ga A|$ of $\d_\Ga A$ on any component in $G \ti G$ is exactly the set $|\d_\Ga A| = A^{\le 1} \cap (G \sm A)^{\le 1}$.
\end{defis}

 Recall that $\U^G$ is the set of all \enquote{compactly almost invariant} subsets of $G$ and that $\B$ is the family of bounded subsets of $G$ (see Notations \ref{nots26}).

\begin{prp}\label{UGaUG}\ 
 \begin{itemize}
  \item $G$ is $\Ga$-connected.
  \item Any bounded subset of $G$ is contained in a compact $\Ga$-connected subset.
  \item The family $\U_{\Ga} = \{ A \subs G ~|~ |\d_\Ga A| \in \B \}$ coincides with $\U^G$.
 \end{itemize}
\end{prp}

 In \cite{AbSK} the third point is stated but it is proved only in the case where $G$ is discrete.

\begin{proof}
 Let $g, h \in G$. As $F$ generates $G$, there exists $n \in \N$ and $f_1, \dots, f_n \in F$ such that $g^{-1} \po h = f_1 \dots f_n$. This means that $1, ~f_1, ~f_1 \po f_2, ~\dots~, ~f_1 \po \dots \po f_n = g^{-1} \po h$ is a $\Ga$-path connecting $1$ to $g^{-1} \po h$, so that the \enquote{translated} path $g = g \po 1, ~g \po f_1, ~\dots~, ~g \po (g^{-1} \po h) = h$ is a $\Ga$-path connecting $g$ to $h$. We have the first point.

 Recall that $F$ is a compact symmetric neighborhood of $1 \in G$. If $K$ is a fixed compact $\Ga$-connected subset of $G$ (for instance a singleton), we easily check that $(K^{\le n})_{n \in \N}$ is an increasing sequence of compact $\Ga$-connected subsets of $G$ whose interiors cover $G$, so that any bounded subset of $G$ is contained in some $K^{\le n}$. We deduce the second point.

\medskip

 Let us show that $U^G = U_\Ga$. If $A \in \U^G$, we check that $|\d_{\Ga} A| \subs (A + A \po F) \po F \in \B$ (recall that $+$ stands for the symmetric difference) hence $A \in \U_\Ga$. Conversely, let $A \in \U_\Ga$. We first check that $A + A \po F \subs |\d_{\Ga} A| = \bigcup_{f \in F} (A + A \po f)$ hence $A + A \po F \in \B$ and $A + A \po f \in \B$ for all $f \in F$. Using the fact that $F$ generates $G$, we get that $A + A \po g \in \B$ and $A + A \po F \po g \in \B$ for all $g \in G$. If $W$ is a non-empty compact subset of $G$, there exists a finite subset $V \subs W$ such that $W \subs \bigcup_{v \in V} F \po v$, and we can show that for any $w \in W$ we have $A + A \po W \subs \bigcup_{v \in V} (A + A \po F \po v) \cup (A + A \po w) \in \B$. As $A \po F$ (resp. $(G \sm A) \po F$) is a neighborhood of $A$ (resp. of $G \sm A$) in $G$, we get $\d_G A \subs |\d_\Ga A|$ so that $\d_G A \in \B$. We conclude that $A \in \U^G$.
\end{proof}

 Until the end of the section, we fix a Specker compactification $X$ of $G$ (Definitions \ref{dSc}) and we write $\La = X \sm G$.

 The following lemma is proven in \cite[Lemma 3.1]{AbSK} :

\begin{lem}
 The map $\fonct{\hat L}{G \ti X}{X}{(g,x)}{\hat {L_g} (x)}$ is continuous.
\end{lem}

 We need the following lemma (generalizing Lemma \ref{blc}) to prove the convergence property. It is adapted from \cite[3.5 and 4.4]{AbEReT} and \cite[3.4]{AbSK}.

\begin{lem}\label{blS}
 Let $\la \in \La$. 
 \begin{itemize}
  \item The family $B_{\la} = \{ V \in \Ne_X (\la) ~|~ |\d_\Ga (V \cap G)| \in \B ~\text{and}~ G \sm V ~\text{is}~ \Ga \text{-connected} \}$ is a local basis at $\la$. 
  \item If $V \in B_{\la}$ and if $B \subs X$ satisfies $|\d_\Ga (B \cap G)| \in \B$, there exists a neighborhood $W$ of $\la$ in $X$ such that for all $g \in W \cap G$, either $g \po B \subs V$ or $g \po B \sups X \sm V$.
 \end{itemize}
\end{lem}

\begin{proof}
 In the proof, we will use the following notations: $V$ and $V'$ will be subsets of $X$, $U$ and $U'$ will be their traces of $G$, i.e. $U = V \cap G$ and $U' = V' \cap G$, and $Q$ and $Q'$ will be their complement in $G$, i.e. $Q = G \sm V = G \sm U$ and $Q' = G \sm V' = G \sm U'$.

  Recall that Proposition \ref{class} gives $\U (X) \subs \U^G$ and that Proposition \ref{UGaUG} gives $\U_\Ga = \U^G$. We deduce that for any subset $B \subs X$, if $\d_X B \subs G$ then $|\d_\Ga (B \cap G)| \in \B$.

\medskip

 Let $V' \in \Ne_X (\la)$, by Lemma \ref{blc} we can assume that $\d_X V' \subs G$ so that $|\d_\Ga U'| \in \B$. There exists a compact $\Ga$-connected subset $K \subs G$ containing $|\d_\Ga U'|$ by Proposition \ref{UGaUG}, we claim that $V = V' \sm K$ belongs to $B_{\la}$.

 As $K$ is a compact subset of $G$ we already have $V \in \Ne_X (\la)$. For all $A, B \subs G$ we check that $(A \cap B) \po F \subs A \po F \cup B \po F$ so that $|\d_\Ga (A \cap B)| \subs |\d_\Ga A| \cup |\d_\Ga B|$. Noting then that $|\d_\Ga (G \sm K)| = |\d_\Ga K| \subs K^{\le 1} \in \B$ and that $U = U' \cap (G \sm K)$, we get that $|\d_\Ga U|$ is bounded.  

 In order to show that $Q = G \sm U$ is $\Ga$-connected, we show that every $\Ga$-path connecting points of $Q$ and \enquote{going out $Q$} can be \enquote{diverted} so that it stays in $Q$. Let then $a, b \in Q$, $u,v \in U = G \sm Q$, $n \in \N$ and $g_1, \dots, g_n \in G$ be such that $(a,u), (u,g_1), \dots, (g_n,v), (v,b) \in \Ga$, we want to show that $a, b \in K$. Note that $Q = Q' \cup K$ and that if $a \in Q'$, as $(a,u) \in \Ga$ and as $Q' = G \sm U' \subs (G \sm U') \po F$ there exists $f \in F$ such that $u = a \po f$ hence $a = u \po f^{-1} \in U \po F \subs U' \po F$ and finally $a \in |\d_\Ga U'| \subs K$. In all cases we have $a \in K$ and the same holds for $b$. As $K$ is $\Ga$-connected there exists a $\Ga$-path connecting $a$ to $b$ in $K$ hence in $Q$.

 Let now $a,b$ be any points in $Q$. By $\Ga$-connectedness of $G$ there exists a $\Ga$-path connecting $a$ to $b$ and we can \enquote{divert all $\Ga$-subpaths going out $Q$} as earlier (there are only finitely many of them) so that the resulting path is a $\Ga$-path in $Q$. We deduce that $Q$ is $\Ga$-connected and the first point.

\medskip

 Let $V \in B_{\la}$ and let $B \subs X$ be such that $|\d_\Ga A|$ is bounded, where $A = B \cap G$. Recall that $\hat R$ is continuous on $G \ti X$ and that $\hat R (G \ti \{ \la \}) = \{ \la \}$. For $g \in \ad {|\d_\Ga A|}$ there exists then $O_g \in \Ne_G (g)$ and $W_g \in \Ne_X (\la)$ such that $\hat R (O_g \ti W_g) \subs V$, and by compacity there exists a finite family $D \subs \ad {|\d_\Ga A|}$ such that $\ad {|\d_\Ga A|} \subs \bigcup_{g \in D} O_g$. Then $W = \bigcap_{g \in D} W_g$ is a neighborhood of $\la$ in $X$ such that for any $g \in W \cap G$ we have $\hat R (|\d_\Ga A| \ti \{ g \}) \subs V$ i.e. $g \po |\d_\Ga A| \subs V$.

 Let us assume that there exists $g \in W \cap G$ and $a, b \in G \sm V$ such that $\{ g^{-1} (a), g^{-1} (b) \}$ intersects both $A$ and $G \sm A$. As $G \sm V$ is $\Ga$-connected, there exists $n \in \N$ and a $\Ga$-path $a = x_0, ~x_1, ~\dots~, ~x_n = b$ in $G \sm V$. The translated path $g^{-1} \po x_0, ~\dots~, ~g^{-1} \po x_n$ connects a point in $A$ and a point in $G \sm A$, it follows from the definition of $|\d_\Ga A|$ that there exists $i$ such that $g^{-1} \po x_i \in |\d_\Ga A|$, but then $x_i = g \po (g^{-1} \po x_i) \in G \sm V \cap g \po |\d_\Ga A| = \vi$, a contradiction.

 This means that for all $g \in W \sm G$ there are two cases: either $g^{-1} (G \sm V) \subs G \sm A$, hence $G \sm V \subs g \po (G \sm A)$ and $g \po A \subs G \cap V \subs V$; or $g^{-1} (G \sm V) \subs A$, hence $g \po A \sups G \sm V$. In both these inclusions we can change $A$ to $B$ and $G$ to $X$ because $x \ma g \po x$ is an homeomorphism of $X$ letting $G$ and $\La$ invariant and because $P \cap \La = \ad {P \cap G} \cap \La$ for all $P \in \U^+$, hence for $P \in \{ B, V, X \sm V \}$ (see the proof of Proposition \ref{class}). We have the second point.
\end{proof}

We can now prove Theorem D:

\begin{thm}\label{thmC+}
 If $G$ is a locally compact compactly generated topological group and if $X$ is a Specker compactification of $G$ then the action of $G$ on $X$ is convergence and its limit set is the boundary $\La = X \sm G$.
\end{thm}

\begin{proof}
 Let $(g_n)_n$ be a wandering net on $G$. As $X$ is compact, up to passing to a subnet we can assume that there exists points $\la, \mu \in X$ such that $g_n \tend{n} \la$ and $g_n^{-1} \tend{n} \mu$ in $X$. The net $(g_n)_n$ is wandering on $G$ which is open in $X$ and $g \ma g^{-1}$ is continuous on $G$, we then have necessarily $\la, \mu \in \La$. We claim that $\clu{g_n}{n}{X}{\mu}{\la}$, we deduce the convergence property.
 
  Assume that it is not the case: there exists a compact set $K \subs X \sans{\mu}$, a neighborhood $V \in \Ne_X (\la)$ and a subnet $(g_i)_i$ such that $g_i \po K \nsubs V$ for all $i$. We can assume that $|\d_\Ga (V \cap G)|$ is bounded and that $G \sm V$ is $\Ga$-connected by Lemma \ref{blS}. There exists then a net $(x_i)_i$ in $K$ such that $g_i \po x_i \notin V$ for all $i$, and up to passing to a subnet we can assume that $(x_i)_i$ converges to a point $x \in K$. 

 Let us fix $h \in G \sm V$. We show that $g_i^{-1} \po h = \hat R (h,g_i^{-1}) \tend i x \in X$, which is in contradiction with $g_i^{-1} \tend{i} \mu$ in $X$ and the continuity of $\hat R$ on $G \ti X$ (note that $x \neq \mu$).

 Let $V_x \in \Ne_X (x)$, we can assume that $|\d_\Ga (V_x \cap G)| \in \B$ (Lemma \ref{blS}). Now, if $W \in \Ne_X (\la)$ is as in the second point of the lemma applied to $B = V_x$, for $i$ large enough $g_i \in W \cap G$ and $x_i \in V_x$ but $g_i \po x_i \notin V$ hence $g_i \po V_x \sups X \sm V$, so that $h \in g_i \po V_x$ and $g_i^{-1} \po h \in V_x$ for $i$ large enough. This means that $g_i^{-1} \po h \tend i x$.

\medskip

 The statement concerning the limit set is not hard. Any $g \in G$ has a compact neighborhood $K$ in $G$ hence in $X$ and the set $\{ h \in G ~|~ h \po K \cap K \neq \vi \}$ is exactly $K \po K^{-1}$ which is compact, so that the ordinary set contains $G$. Conversely, as $G$ is dense in $X$ for any $\la \in \La$ there exists a net $(g_n)_n$ in $G$ such that $g_n \tend{n} \la$, up to passing to a subnet we can assume that there exists $\mu \in X$ such that $g_n^{-1} \tend{n} \mu$, and as above we have $\mu \in \La$ and $\clu{g_n}{n}{X}{\mu}{\la}$. We deduce that the limit set contains $\La$ and the stated result.
\end{proof}

 Note that B. Bowditch gives in \cite{BoGCSESG} a simple proof that the action of a finitely generated group on its space of ends is convergence, this proof can be extended without difficulty to our situation. We easily deduce that the action of $G$ on any Specker compactification is convergence, but is seems hard to prove directly that such an action is minimal. For instance, there exists locally compact groups having infinitely many ends, one of them being fixed by the whole group (see \cite{AbPF}).

\medskip

 Let us now prove Corallary E. We first show a lemma establishing the existence of a \enquote{minimal Specker pullback}:

\begin{lem}\label{Spb}
 Let $X$ and $Y$ be two quasi-Specker compactifications of $G$. Let $\Sr (X,Y)$ be the closure in $X \ti Y$ of the diagonal $\De (G) = \{ (g,g) ~|~ g \in G \}$, we write $\pi_1: X \ti Y \ra X$ (resp. $\pi_2: X \ti Y \ra Y$) the canonical projection. We have the following:
  \begin{itemize}
   \item The space $\Sr (X,Y)$ is a quasi-Specker compactification of $G$ (for the natural $\hat L$ action).
   \item For every quasi-Specker compactification $Z$ of $G$ endowed with two $G$-equivariant continuous maps $\pi_X: Z \ra X$ and $\pi_Y: Z \ra Y$ extending the identity on $G$ there exists a unique surjective $G$-equivariant continuous map $\phi: Z \ra \Sr (X,Y)$ extending the identity on $G$ and such that $\pi_1 \ci \phi = \pi_X$ and $\pi_2 \ci \phi = \pi_Y$.
   \item When $X$ and $Y$ are \enquote{complete} Specker compactifications of $G$, so is $\Sr (X,Y)$.
  \end{itemize}
\end{lem}

\begin{proof}
 Note that $\De (G)$ is naturally homeomorphic to $G$, we identify then $g \in G$ and $(g,g) \in \De (G)$.

 By construction $\Sr (X,Y)$ is a compactum and contains $G$ as a dense subset, we have the points 1 and 2 in Definitions \ref{dSc}.

 Let us write $\hat R_X$ and $\hat L_X$ (resp. $\hat R_Y$ and $\hat L_Y$) the extensions to $G \ti X$ (resp. to $G \ti Y$) of the right and left translations of $G$. By definition of the product topology, the maps $$\fonct{\hat R}{G \ti (X \ti Y)}{X \ti Y}{(g,(x,y))}{(\hat R_X (g,x),\hat R_Y (g,y))} \et \fonct{\hat L}{G \ti (X \ti Y)}{X \ti Y}{(g,(x,y))}{(\hat L_X (g,x),\hat L_Y (g,y))}$$ are continuous. Using the fact that $G$ is dense in $\Sr(X,Y)$ it is an exercise to show that these two maps let $\Sr (X,Y)$ invariant.

 We deduce the points 4 and 5 of Definitions \ref{dSc}: $\Sr (X,Y)$ is a quasi-Specker compactification of $G$.

 It is easy to see that $\Sr (X,Y) \sm G \subs (X \sm G) \ti (Y \sm G)$. As the connected components of a product are the product of the components and as the components of a subspace are contained in the components of the space, we deduce the last point of the lemma.

 Let now $Z$ be a quasi-Specker compactification of $G$ endowed with $G$-equivariant continuous maps $\pi_X: Z \ra X$ and $\pi_Y: Z \ra Y$ extending the identity on $G$. The map $\fonct{\phi}{Z}{\Sr(X,Y)}{t}{(\pi_X (t),\pi_Y (t))}$ is the only map $Z \ra \Sr (X,Y)$ satisfying $\pi_X = \pi_1 \ci \phi$ and $\pi_Y = \pi_2 \ci \phi$. By construction the map $\phi$ is continuous, $G$-equivariant and extend the identity on $G$, one can also show that it is surjective.
\end{proof}

 Recall that a \gr{pullback} of two minimal convergence actions of $G$ on compacta $M$ and $N$ is a triple $(P,\pi_M,\pi_N)$ where $P$ is a compactum on which $G$ acts with minimally convergence property and $\pi_M: P \ra M$ (resp. $\pi_N: P \ra N$) is a $G$-equivariant surjective continuous map. The pullback $(P,\pi_M,\pi_N)$ is \gr{minimal} if for any other pullback $(T,p_M,p_N)$ there exists a $G$-equivariant continuous map $\phi: T \ra P$ such that $\pi_M \ci \phi = p_M$ and $\pi_N \ci \phi = p_N$.

 We can now prove Corollary E:

\begin{thm}\label{thmD}
 Any two minimal convergence actions of a locally compact compactly generated topological group on totally disconnected compacta containing more than two points admit a minimal pullback.
\end{thm}

\begin{proof}
 Assume that $G$ is compactly generated and acts minimally with convergence property on two totally disconnected compacta $M$ and $N$ containing more than two points. We first construct a pullback and then show that it is minimal.

 By Corollary \ref{corthmB} the attractor sums $X = G \sqcup M$ and $Y = G \sqcup N$ are Specker compactifications of $G$. We consider the Specker compactification $\Sr (X, Y)$ of $G$ and the projections $\pi_1: \Sr (X,Y) \ra X$ and $\pi_2: \Sr (X,Y) \ra Y$ given by Lemma \ref{Spb}. We set $P = \Sr (X,Y) \sm G$ and $\pi_M = (\pi_1)_{|P}$ (resp. $\pi_N = (\pi_2)_{|P}$), it is a $G$-equivariant continuous map $P \ra M$ (resp. $P \ra N$). By Corollary \ref{thmC} the action of $G$ on $P$ is convergence and minimal so that $(P,\pi_M,\pi_N)$ is a pullback of the actions of $G$ on $M$ and $N$.

 Let now $(T,p_M,p_N)$ be an other pullback of these actions. We first consider the totally disconnected quotient space $T'$ of $T$ by the relation \enquote{belong to the same component}, the natural action of $G$ on $T'$ is convergence and minimal (see Lemma \ref{qcc}, we have here $T = \La$).

 For every connected component $C$ of $T$ the image $p_M (C)$ is a connected subset of $M$ (the map $p_M$ is continuous) hence a singleton ($M$ is totally disconnected). There exists then a map $p'_M : T' \ra M$ satisfying $p'_M \ci \pi = p_M$ (where $\pi : T \ra T'$ is the canonical projection). 

 The construction of $p'_M$ and the $G$-equivariance of $p_M$ and $\pi$ induce that $p'_M$ is also $G$-equivariant, it is not hard to see that it is continuous. Applying the same arguments with $N$ we get a $G$-equivariant surjective continuous map $p'_N : T' \ra N$ such that $p'_N \ci \pi = p_N$.

 It follows that $(T', p'_M, p'_N)$ is a pullback of the actions of $G$ on $M$ and $N$. Note that $|T'| \geq 3$ (the action of $G$ on $M$ is minimal, $|M| \geq 3$ and $p'_M$ is $G$-equivariant) and that $p'_m$ is not constant (it is surjective).

  Let us then consider the attractor sum $Z = G \sqcup T'$, by Corollary \ref{corthmB} again this is a Specker compactification of $G$. Let $p'_X: Z \ra X$ (resp. $p'_Y: Z \ra Y$) be the extension of $p'_M$ (resp. of $p'_N$) on $Z$ by the identity of $G$, this map it continuous, surjective and $G$-equivariant (see Lemma \ref{cGeAS}). By the second point of Lemma \ref{Spb} there exists then a surjective $G$-equivariant continuous map $\psi : Z \ra \Sr (X,Y)$ extending the identity on $G$ and satisfying $\pi_1 \ci \psi = p'_X$ and $\pi_2 \ci \psi = p'_Y$.

 We have then $\psi (T') = P$ so that the map $\psi ~ \ci ~ \pi: T \ra P$ is $G$-equivariant, surjective and continuous. We have the stated result.
\end{proof}

 We end the section with considerations about the hypothesis in Theorem \ref{thmC+}. The following two facts are known:

\begin{itemize}
  \item there exists two minimal convergence actions of the discrete free group of countable rank on two totally disconnected compacta admitting no pullback space (\cite[Prop 5.5]{GeLPSRHAG}),
  \item there exists two minimal convergence actions of the discrete free group of rank $3$ on two compacta (one of them not being totally disconnected) admitting no pullback space (\cite[Prop 5.2]{GeLPSRHAG}).
\end{itemize}

 The last result follows directly from an example due to O. Baker and T. Riley in \cite{BaRiCTMDNAE}: the authors exhibit a hyperbolic group $G$ and a free subgroup $H < G$ of rank $3$ such that the inclusion do not induce a Cannon-Thurston map $\d H \ra \d G$. Note that in \cite{GeLPSRHAG} the authors study 2-cocompact actions (i.e. geometrically finite), but we do not need this property in our paper.

 This two results show that both hypothesis \enquote{the group is compactly generated} and \enquote{the Specker boundary is totally disconnected} cannot be removed from the statement of Theorem \ref{thmC+}.

\section{The ends of a compactly generated group}\label{s8}

 The aim of this section is to give an other interpretation of the space of ends of a compactly generated group. This construction is inspired by a remark of V. Gerasimov (in the context of a finitely generated discrete group) \cite[4.1]{GeFMRHG}.

 We first recall some results about uniform spaces. Recall that a filter on a set $X$ is a non-empty family $\F$ of subsets of $X$ stable by finite intersection ($\fa A,B \in \F, ~A \cap B \in \F$) and by \enquote{growth} (if $A \in \F$ and if $B \subs X$ satisfies $A \subs B$ then $B \in \F$).

\begin{defi}
 A \gr{uniformity} $\V$ on a set $X$ is a filter of $X \ti X$ such that:
   \begin{itemize}
    \item for all $V \in \V$ the set $V^{-1} = \{ (y,x) ~|~ (x,y) \in V \}$ belongs to $\V$,
    \item every $V \in \V$ contains the \gr{diagonal} $\De = \{ (x,x) ~|~ x \in X \}$,
    \item for all $V \in \V$ there exists $U \in \V$ such that $V$ contains the set $$U \ci U = \{ (x,z) \in X \ti X ~|~ \ex y \in X, ~ (x,y),(y,z) \in U \}.$$
   \end{itemize}
  The uniformity $\V$ is \gr{exact} if $\bigcap \V = \De$, i.e. $\fa x \neq y \in X, \ex V \in \V, ~(x,y) \notin \V$.
\end{defi}

\begin{dfpp}
 Let $\V$ be a uniformity on $X$.
 \begin{itemize}
  \item For $x \in X$ and $V \in \V$ let $V[x] = \{ y \in X ~|~ (x,y) \in V \}$. There exists a unique topology on $X$ (the \gr{uniform topology}) such that the family $(V[x])_{V \in \V}$ is a local basis at every $x \in X$. This topology is Hausdorff if and only if $\V$ is exact.
  \item For $V \in \V$ we say that a subset $A \subs X$ is \gr{$V$-small} if $A \ti A \subs V$ and we write $\Sm (V)$ the set of all $V$-small subsets of $X$.
  \item A \gr{Cauchy filter} on $X$ is a filter on $X$ containing $V$-small sets for every $V \in \V$. A Cauchy filter is \gr{minimal} if it is minimal for the inclusion, we write $\tilde X$ the set of all minimal Cauchy filters on $X$. The filters $\Ne_X (x)$ for $x \in X$ are minimal Cauchy filters, they are called \gr{convergent}.
  \item For $V \in \V$ we set $\tilde V = \{ (\F,\G) \in \tilde X \ti \tilde X ~|~ \F \cap \G \cap \Sm(V) \neq \vi \}$, the family $\{ \tilde V ~|~ V \in \V \}$ generates a uniformity on $\tilde X$. This uniform space is called the \gr{completion} of $X$ by the uniformity $\V$. When $\V$ is exact the map $x \ma \Ne_X (x)$ is a homeomorphism of $X$ on a dense open subset of $\tilde X$ (both spaces endowed with the uniform topology) and we can identify any $x \in X$ with $\Ne_X (x) \in \tilde X$.
  \end{itemize}
\end{dfpp}

 In the whole section we fix a locally compact compactly generated topological group $G$ and we use the notations and results of the previous section. The Cayley graph $\Ga$ of $G$ is seen as a subspace of $G \ti G$. Recall that $\B$ denotes the set of all bounded subsets of $G$ and let $\B_\Ga$ be the set of all bounded subsets of $\Ga$ for this topology.

\medskip

 Let us fix a subset $M \subs \Ga$. A \gr{$M$-path} is a sequence $g_0, \dots, g_n$ (where $n \in \N$) such that $(g_i,g_{i+1}) \in M$ for all $i$. A subset $A \subs G$ is \gr{$M$-connected} if any two points in $A$ can be connected with a $M$-path. The \gr{$M$-components} of $G$ are the maximal connected subsets of $G$, we write $\pi_0 (M)$ the set of all the $M$-components.

 We can now define a \enquote{visibility} on $G$.

\begin{dfpp}
 ~\newline
 \indent For $M \in \B_\Ga$ let $v_M = \{ (g,h) \in G \ti G ~|~ \text{there exists a } (\Ga \sm M)\text{-path connecting } g \et h \}$. The family $\V = \{ v \subs G \ti G ~|~ \ex M \in \B_\Ga, ~v_M \subs V \}$ is an exact uniformity on $G$.
\end{dfpp}

\begin{proof}
 For all $M \in \B_\Ga$ we have $v_M = \bigcup_{C \in \pi_0 (\Ga \sm M)} C \ti C$ so that $\De \subs v_M = v_M^{-1} = v_M \ci v_M$, moreover for all $M,N \in \B_\Ga$ we have $v_M \cap v_N \sups v_{M \cup N}$. We deduce that $\V$ is a uniformity on $G$.

 This also shows that a subset $A \subs G$ is $v_M$-small if and only if there exists some $C \in \pi_0 (\Ga \sm M)$ such that $A \subs C$.

\medskip

 For all $g \in G$ the set $F(g) = \{ (h_1,h_2) \in \Ga ~|~ h_1 = g \text{ or } h_2 = g \}$ is bounded and the component of $\Ga \sm F(g)$ to which $g$ belongs is exactly $\{ g \}$. We deduce that $\V$ is exact.
\end{proof}

 Note that the topology induced by $\V$ on $G$ is the discrete topology.

 For $M \in \B_\Ga$ let $u_M$ be the set of all $(g,h) \in G \ti G$ such that every shortest $\Ga$-path connecting $g$ and $h$ avoids $M$. As $G$ is $\Ga$-connected we have $u_M \subs v_M$ for all $M \in \B_\Ga$. This is the definition of a visibility on $\Ga$, see \cite{GeFMRHG}.

\medskip

The main result of this section is the following theorem (see Propositions \ref{class} and \ref{ends} for the definitions):

\begin{thm}\label{thmE}
 For every non-convergent minimal Cauchy filter $\F$ on $G$ the family $\F \cap \U^G$ is an unbounded ultrafilter on $\U^G$, i.e. a point of the set of ends $E(G)$ of $G$.

 The map $\fonct{f}{\tilde G \sm G}{E(G)}{\F}{f(\F) = \F \cap \U^G}$ is a homeomorphism.
\end{thm}

\begin{proof}
 Let $\F$ be a non-convergent Cauchy filter on $G$, we first show that $f(\F) \in E(G)$. The points (F0), (F1) and (F2) of the proof of Proposition \ref{class} are immediate.

 Assume that $\F$ contains a bounded subset $B \in \B$. The set $F(B) = \{ (h_1,h_2) \in \Ga ~|~ h_1 \in B \text{ or } h_2 \in B \}$ is bounded and $\F$ is Cauchy so that there exists $A \in \Sm (v_{F(B)}) \cap \F$. We have then $A \cap B \in \F$, but $F(B)$ contains all the edges with one end in $B$ so that necessarily $A \cap B$ is a singleton and $\F$ is convergent. We deduce the point (NB).

 In order to show the point $(U)$ we need to prove the following fact: every $A \subs G$ such that $\d_\Ga A \in \B_\Ga$ is a union of $(\Ga \sm \d_\Ga A)$-components (and so is $G \sm A$). Indeed, let $a \in A$ and $b \in G$ be in the same $(\Ga \sm \d_\Ga A)$-component. There exists then a $(\Ga \sm \d_\Ga A)$-path connecting $a$ and $b$, and if some point on this path is not in $A$ then this path has an edge in $\d_\Ga A$, a contradiction.

 For all $A \subs G$ such that $\d_\Ga A \in \B_\Ga$, as every $P \in \Sm(v_{\d_\Ga A})$ is contained in a $(\Ga \sm \d_\Ga A)$-component every such $P$ is a subset of $A$ or of $G \sm A$.

 Let then $A \in \U^G$, recall that $\d_\Ga A \in \B_\Ga$ (see Proposition \ref{UGaUG}). As $\F$ is Cauchy it contains a $(v_{\d_\Ga A})$-small set $P$, we have then $P \subs A$ or $P \subs G \sm A$ so that $A \in \F$ or $G \sm A \in \F$. We deduce the point $(U)$.

\medskip

 The following lemma, inspired by \cite[I,p.117,ex.19]{BouTG}, enables us to give the inverse map of $f$.

\begin{lem}
 Let us fix $M \in \B_\Ga \sm \{ \vi, G \}$, we write $p(M) = \{ g \in G ~|~ \ex h \in G, ~(g,h) \in M \text{ or } (h,g) \in M \}$ (we have $\vi \neq p(M) \in \B$). Let $K$ be a compact subset of $G$ such that $p(M)^{\le 1} \subs \int K$ (for instance $K = \{ 1 \}^{\le n}$ for some $n \in \N$ big enough, see Proposition \ref{UGaUG}) and $L = K^{\le 1} \sm \int K$, this is a compact subset.
 \begin{enumerate}
  \item For all $C \in \pi_0 (\Ga \sm M)$ we have $C \cap p(M)^{\le 1} \neq \vi$.
  \item For all $C \in \pi_0 (\Ga \sm M)$ such that $C \cap L = \vi$ we have $C \subs \int K$.
  \item The set $\{ C \in \pi_0 (\Ga \sm M) ~|~ C \cap L \neq \vi \}$ is finite.
  \item For all $\F \in E(G)$, there exists a unique $\C(M,\F) \in \pi_0 (\Ga \sm M) \cap \F$, moreover $\C(M,\F) \notin \B$.
 \end{enumerate}
\end{lem}

\begin{proof}[Proof of the lemma]
 Let $C \in \pi_0 (\Ga \sm M)$ (then $C \neq \vi$) and assume $C \nsubs p(M)^{\le 1}$. Pick any $c \in C \sm p(M)^{\le 1}$ and any $g \in p(M)$, and let $c = g_0, \dots, g_n = g$ be a $\Ga$-path. The set $\{ i \in \se{1,n} ~|~ g_i \in p(M) \}$ is not empty and has a smallest element $i > 0$, and as $g_0, \dots, g_{i-1} \notin p(M)$ we have $(g_0,g_1), \dots, (g_{i-2},g_{i-1}) \notin M$ so that $g_{i-1} \in C$ (the path from $c = g_0$ to $g_{i-1}$ stays in the same $(\Ga \sm M)$-component, namely $C$). As $g_i \in p(M)$ we have $g_{i-1} \in C \cap p(M)^{\le 1}$, we deduce the first point.

 Assume that $C \in \pi_0 (\Ga \sm M)$ satisfies $C \cap L = \vi$. Fix $c_0 \in C \cap p(M)^{\le 1}$ (then $c_0 \in C \cap \int K$) and let $c$ be any point in $C$. By definition of a $(\Ga \sm M)$-component there exists a $\Ga$-path $c_0, \dots, c_n = c$ in $C$. Assume that this path intersects $G \sm \int K$ and let $i > 0$ be the smallest index such that $c_i \notin \int K$, we have then $c_i \in C \cap (\int K^{\le 1} \sm \int K) \subs C \cap L = \vi$, a contradiction. As a consequence we have $c \in \int K$, we deduce the second point.

 Let $C \in \pi_0 (\Ga \sm M)$ be such that $C \cap L \neq \vi$, we first show that $C \cap L$ is open in $L$. Let $c \in C \cap L$.  If $g \in L$ is such that $(c,g) \in \Ga$, as $c \notin p(M)$ we have $(c,g) \notin M$ so that $g \in C$. The set $c \po F$ is an open neighborhood of $c$ in $G$ so that $(c \po F) \cap L$ is a neighborhood of $c$ in $L$, contained in $C \cap L$.

 As $L$ is compact, the open covering $L = \bigcup_{C \in \pi_0 (\Ga \sm M)} (C \cap L)$ has a finite subcovering, and as different components are disjoints, we have the third point.

 Let now $\F \in E(G)$. Every $C \in \pi_0 (\Ga \sm M)$ such that $C \cap L = \vi$ is contained in $K$ by the third point, so that $X$ is covered by the finite union $K \cup \bigcup \{ C \in \pi_0 (\Ga \sm M) ~|~ C \cap L \neq \vi \}$. Note that every $C \in \pi_0 (\Ga \sm M)$ satisfy $\d_\Ga C \subs M$ so that $C \in \U^G$, as $X \in \F$ the properties (U) and (NB) give the fourth point.
\end{proof}

 The last point of the lemma is in fact also true for any non-convergent Cauchy filter $\F$ on $G$, indeed $\F$ must contain a $v_M$-small set, hence a (unbounded) $(\Ga \sm M)$-component, for every $M \in \B_\Ga$. Note then that $\C(M,f(\F)) = \C(M,\F)$ for all $M \in \B_\Ga$.

\medskip

 Let $\F \in E(G)$, we define $g(\F) = \{ A \subs G ~|~ \ex M \in \B_\Ga, ~\C(M,\F) \subs A \}$. For all $M,N \in \B_\Ga$ we have $\C(M,\F) \cap \C(N,\F) \subs \C(M \cup N, \F)$, we easily deduce that $g(\F)$ is a filter on $G$. As the $\C(M,\F)$ are $v_M$-small and unbounded this filter is Cauchy and non-convergent.

 Let $\F_0$ be the unique minimal Cauchy filter contained in $g(\F)$ (see \cite[II]{BouTG}), we have $\F_0 \subs g(\F)$. As different component are disjoint we necessarily have $\C( M, \F_0) = \C(M, \F)$ so that $g(\F) \subs \F_0$ and the equality. We deduce that $g(\F) \in \tilde G$.

 This shows in fact that for every Cauchy filter $\A$ on $G$, the family $(\C(M,\A))_{M \in \B_\Ga}$ is a basis for the unique minimal Cauchy filter contained in $\A$.

 Note that for every $\F \in E(G)$ and every $A \in \U^G$, we have $A \in \F$ if and only if $\C(\d_\Ga A, \F) \subs A$ (recall that either $\d_\Ga A \subs A$ or $\d_\Ga A \subs G \sm A$ and that either $A \in \F$ or $G \sm A \in \F$). The previous construction \enquote{preserves} the family $(\C(M,\po))_{M \in \B_\Ga}$, we deduce that for all $A \in \U^G$ we have $A \in \F$ if and only if $A \in f(g(\F))$. This shows that $\F = f(g(\F))$.

 Similarly, let $\F$ be a minimal Cauchy filter on $G$, then $\F$ and $g(f(\F))$ are minimal Cauchy filters on $G$ with the \enquote{same family $(\C(M,\po))_{M \in \B_\Ga}$}, they are then equal.

 We have shown that the map $f$ is a bijection and that $g$ is its inverse.

\medskip

 Let us show that $f$ is continuous: let $\F \in \tilde G \sm G$ and let $O$ be a neighborhood of $f(\F)$ in $E(G)$, we show that $f^{-1} (O)$ is a neighborhood of $\F$ in $\tilde G \sm G$. We can assume that $O = \{ \G \in E(G) ~|~ A \in \G \}$ for some open subset $A \in \U^G$, note then that $A \in f(\F) = \F \cap \U^G \subs \F$. Let us write $M = \d_\Ga A$ and $V = \tilde {v_M}[\F] \sm G = \{ \G \in \tilde G \sm G ~|~ (\F,\G) \in \tilde {v_M} \}$, it is a neighborhood of $\F$. For all $\G \in \tilde G \sm G$ we have $\G \in V \LR \F \cap \G \cap \Sm(v_M) \neq \vi \LR \C(M,\F) = \C(M,\G)$, and as $A \in \F$ we deduce that $\G \in V$ imply $A \in \G$ and $f(\G) \in O$, we have then $\F \in V \subs f^{-1}(O)$.

 Conversely, let $\F \in \tilde G \sm G$ and let $U$ be a neighborhood of $\F$ in $\tilde G \sm G$, we can assume that there exists $M \in \B_\Ga$ such that $U = \tilde {v_M} [\F] \sm G$. Let us write $C = \C(M,\F)$, we can show that $|\d_\Ga \int C| \subs |\d_\Ga C|^{\le 1}$ so that $\int C \in \U^G$, the set $O = \{ \G \in E(G) ~|~ \int C \in \G \}$ is then a neighborhood of $f(\F)$ in $E(G)$. As previously we show that for all $\G \in \tilde G \sm G$ such that $f(\G) \in O$ we have $\G \in U$, we conclude that $f(\F) \in O \subs f(U)$ and that $f$ is open.

 Finally, $f$ is a homeomorphism as stated. 
\end{proof}

 For any $M \subs N \in \B_\Ga$ let us consider the natural map $\io_{M,N}: \pi_0 (\Ga \sm N) \ra \pi_0 (\Ga \sm M)$ which sends a $(\Ga \sm N)$-component to the $(\Ga \sm M)$-component containing it, we have then an inverse system $(\{ \pi_0 (\Ga \sm M) ~|~ M \in \B_\Ga \}, \{ \io_{M,N} ~|~ M \subs N \in \B_\Ga \})$. It follows from the previous proof that the map $\F \ma (\C(M,\F))_{M \in \B_\Ga}$ is a bijection from the set $\tilde G \sm G$ of the minimal non-convergent Cauchy filters to the inverse limit of this inverse system.

 Theorem \ref{thmE} can then be rephrased in the following \enquote{unformal} way: the set of ends of a compactly generated group coincides with the sets of the components of the \enquote{ideal boundary} of its Cayley-graph.

\bibliographystyle{plain}
\bibliography{bibliographie}

\end{document}